\documentclass[12pt]{amsart}
\usepackage[utf8]{inputenc}

\usepackage[margin=1.5in]{geometry}
\usepackage{amsmath,amssymb,amsthm}
\usepackage{dsfont}
\usepackage{tikz}
    \usetikzlibrary{cd}

\theoremstyle{plain}
\newtheorem{theorem}{{Theorem}}[section]

\newtheorem{definition}[theorem]{{Definition}}

\newcommand{\Z}{\mathbb{Z}}
\newcommand{\F}{\mathbb{F}} 

\newcommand{\HFKhat}{\widehat{\mathrm{HFK}}}

\usepackage{hyperref}
\hypersetup{
    colorlinks=true,
    linkcolor=purple,
    citecolor=teal,      
    urlcolor=purple,
}

\title{A note on knot Floer thickness and the dealternating number}
\author{Linh Truong}

\begin{document}
\begin{abstract}
In this note, we give a short proof that knot Floer thickness is a lower bound on the dealternating number of a knot.  The result is originally due to work of Abe and Kishimoto, Lowrance, and Turaev. Our proof is a modification of the Stipsicz-Szab\'o approach using Kauffman states to show that thickness bounds the minimal number of bad domains in a knot diagram. 
\end{abstract}
\maketitle

\section{Introduction}
Alternating knots are knots which admit a diagram in which the crossings alternate between over-passes and under-passes. An intrinsic definition of alternating (i.e. independent of a diagrammatic definition) was given by Greene and Howie \cite{Greene, Howie}. Alternating knots include any $T(2,n)$ torus knot and all prime knots of seven or fewer crossings. Several invariants can obstruct a knot from being alternating, such as degree of the Alexander polynomial \cite{Crowell, Murasugi}, the thickness of knot Floer homology \cite{OS-alternating}, thickness of Khovanov homology \cite{Lee}, and the Heegaard Floer homology of its double branched cover \cite{OS-dbc}. Another obstruction to being alternating is the dealternating number, introduced in \cite{Adamsetal}. 

\begin{definition}
The dealternating number of a knot $K$, denoted $\text{dalt}(K)$, is the minimum number of crossing changes to turn a diagram for $K$ into an alternating diagram. 
\end{definition}

Observe that if a knot $K$ is alternating, then the dealternating number of $K$ vanishes.  Note that if $K$ has crossing number $n$, then $0 \leq \text{dalt}(K) \leq n/2$. 

Recently, Stipsicz-Szabo \cite{StipsiczSzabo} show that the $\delta$-thickness  $\text{th}(K) $ of knot Floer homology provides a lower bound on $\beta(K)$, the minimal number of bad domains in a diagram for a knot $K$. A bad domain is one in which there exists an edge whose two vertices are both under-passes or both over-passes (as opposed to one over- and one under-pass). The invariant $\beta(K)$ is a measure of how far a knot is from being alternating, since if $K$ is alternating, then $\beta(K) = 0$. Their proof relies on studying the knot Floer $\delta$-gradings of Kauffman states in a knot diagram, and a similar strategy can be leveraged to show the following theorem. 

\begin{theorem} \label{thm:dalt}
$\text{th}(K) \leq  \text{dalt}(K)$. 
\end{theorem}
We remark that an alternative proof of this theorem already exists: Lowrance \cite{Lowrance} shows the bound $\text{th}(K) \leq g_T(K)$, where $g_T(K)$ denotes the Turaev genus of $K$. Furthermore, Abe and Kishimoto \cite{AbeKishimoto} show  $g_T(K) \leq \text{dalt}(K)$ by using work of Lowrance \cite{Lowrance} and Turaev \cite{Turaev}. We provide a shorter proof of this result, without referring to the Turaev genus, by modifying the strategy of \cite{StipsiczSzabo}.

An analogous result showing the thickness of Khovanov homology bounds the dealternating number is due to Asaeda and Przytycki \cite{AsaedaPrzytycki} (and was reproven by Champanerkar-Kofman in \cite{ChampanerkarKofman}). 

The dealternating number of a knot is difficult to compute, and the calculation for torus knot of most braid indices remains open. With Turner's calculation of the thickness of Khovanov homology for torus knots of braid index three \cite{Turner}, Abe and Kishimoto \cite{AbeKishimoto} determine the dealternating number for all $(3,q)$-torus knots. For torus knots of braid index 5 or fewer, the dealternating number is computed up to an error of at most two \cite{Lowranceetal}. 

\subsection{Acknowledgements}
The author thanks Andr\'as Stipsicz for helpful comments. The author was partially supported by NSF grant DMS-2104309. 

\section{Knot Floer thickness}
Let $K \subset S^3$ be a knot. The hat version of knot Floer homology $\HFKhat(K) $ over $\F = \F_2$ is knot invariant \cite{OS-knot} which is a finite-dimensional bigraded vector space $\HFKhat(K) = \sum_{M, A \in \Z}\HFKhat_M(K, A) $. Here, $M$ denotes the Maslov grading and $A$ denotes the Alexander grading. We can collapse the two gradings into a single grading via $\delta = M -A$. This produces the $\delta$-graded invariant $\HFKhat(K) = \sum_\delta \HFKhat^\delta(K) $, where $\HFKhat^\delta(K) \subset \HFKhat(K)$ is the subspace of homogeneous elements of grading $\delta \in \Z$. The (knot Floer) thickness of $K$ is the thickness of this finite dimensional $\delta$-graded vector space  $\HFKhat(K)$, which by definition is the largest possible difference between $\delta$-gradings of two non-zero homogeneous elements:
\[\text{th}(K) = \max\{ a \in \Z \mid \HFKhat^a(K) \neq 0 \} - \min\{ a \in \Z \mid \HFKhat^a(K) \neq 0 \} . \]

Given a knot diagram $\mathcal{D}$ for $K$, we can define knot Floer homology as the homology of a chain complex, whose construction \cite{OS-alternating} we recall here. Choose an edge $e$ in the diagram $\mathcal{D}$. Associated to the marked diagram $(\mathcal{D}, e)$, the chain complex $C_{\mathcal{D}, e}$ has an underlying $(M,A)$-bigraded vector space  generated over $\F$ by the Kauffman states (described in the next paragraph) of $(\mathcal{D}, e)$. By \cite{OS-alternating}, there exists a boundary map $\partial: C_{\mathcal{D}, e} \to C_{\mathcal{D}, e}$ of $(M,A)$-bidegree $(-1,0)$ with the property that the homology $H_*(C_{\mathcal{D}, e}, \partial)$ is isomorphic to the knot Floer homology $\HFKhat(K)$ of $K$ \cite{OS-knot} (as a bigraded vector space).

Let $\text{cr}(\mathcal{D})$ denote the set of crossings in the diagram $\mathcal{D}$ for the knot $K$. Denote the set of domains which do not contain the edge $e$ on their boundary by $Dom_e(\mathcal{D})$. A Kauffman state $x$ is a bijection $x: \text{cr}(\mathcal{D}) \to Dom_e(\mathcal{D})$ such that for each crossing $c \in \text{cr}(\mathcal{D})$, the domain $x(c)$ is one of the (at most four) domains meeting at $c$. Equivalently, a Kauffman state is a choice of a corner in each domain of the diagram $\mathcal{D}$ that is not adjacent to the special marked edge $e$, such that no two corners belong to the same crossing in $\text{cr}(\mathcal{D})$.  Given a Kauffman state $x$, we will denote the crossing associated to the domain $x(c)$ by $c_x$. 
The Maslov, Alexander, and $\delta$ gradings of a Kauffman state are computed by adding the local contributions at each crossing; see Figure~\ref{fig:gradings} for the $\delta$-grading. 
\begin{figure}
\begin{tikzpicture} 
\node[](s1){};
\node[](s2) at (2cm,0){};
\node[](n1) at (0,2cm){};
\node[](n2) at (2cm,2cm){};
\node[](center) at (1cm,1cm){};
\draw[->] (s1) to  (n2);
\draw[-] (s2) to  (center);
\draw[->] (center) to  (n1);
\node[above] at (center.north){$\frac{1}{2}$};
\node[below] at (center.south){$\frac{1}{2}$};
\node[right] at (center.east){\ $0$ \ };
\node[left] at (center.west){\ $0$ \ };
\end{tikzpicture}
\quad 
\begin{tikzpicture}
\node[](s1){};
\node[](s2) at (2cm,0){};
\node[](n1) at (0,2cm){};
\node[](n2) at (2cm,2cm){};
\node[](center) at (1cm,1cm){};
\draw[->] (s2) to  (n1);
\draw[-] (s1) to  (center);
\draw[->] (center) to  (n2);
\node[above] at (center.north){$-\frac{1}{2}$};
\node[below] at (center.south){$-\frac{1}{2}$};
\node[right] at (center.east){\ $0$ \ };
\node[left] at (center.west){\ $0$ \ };
\end{tikzpicture}
\caption{Local contribution to the $\delta$-grading}
Every Kauffman state $x$ associates to each crossing a domain, which contributes to the $\delta$-grading of $x$ according to this illustration. 
 \label{fig:gradings}
\end{figure}

\begin{proof}[Proof of Theorem \ref{thm:dalt}] Suppose $\mathcal{D}$ is a non-alternating diagram for the knot $K$.
Let $\text{cr}(\mathcal{D})$ denote the set of crossings of $\mathcal{D}$. Choose a special marked edge $e$ in the diagram $\mathcal{D}$ with respect to which the Kauffman states will be defined.

As observed in the proof of \cite[Proposition 2.1]{StipsiczSzabo}, the $\delta$-grading of a Kauffman state $x$ is given by
\[ \delta(x) = - \frac{1}{4} \text{wr}(\mathcal{D}) + \sum_{c \in \text{cr}(\mathcal{D})} f(c_x)\]
where $c_x$ denotes the marked corner (determined by $x$) of the crossing $c$, and $f(c_x) \in \{ \frac{1}{4}, -\frac{1}{4}\}$. Here, $\text{wr}(\mathcal{D})$ denotes the writhe of the diagram. Note that $\sum_{c \in \text{cr}(\mathcal{D})} f(c_x)$ is expressed as a sum of $f$-values over crossings in the diagram $\mathcal{D}$, but via the bijection $x$, we will also find it useful below to view this as a sum of $f$-values over all domains in $Dom_e(\mathcal{D})$. 

Choose a set of $n$ crossings $c_1, \dots, c_n$ in the diagram $\mathcal{D}$ such that after changing these crossings, the new diagram is alternating. We call each $c_i$ a fixable crossing. (Note that there is more than one choice of a set of fixable crossings for a diagram $\mathcal{D}$.) 
Call the remaining crossings in $\text{cr}(\mathcal{D}) \setminus \{c_1, \dots, c_n\}$ the static crossings. 

Let $x$ be a Kauffman state of $\mathcal{D}$. 
Note that given a domain $D \in Dom_e(\mathcal{D})$, the domain $D$ has exactly one corner marked (chosen) by $x$. 

We call a domain good if it is not bad. After changing all the fixable crossings, each domain will be good. In a good domain, all corners have the same $f$-value $a_D  \in \{ \frac{1}{4}, -\frac{1}{4}\}$, where $a_D$ depends on the domain $D$ \cite{StipsiczSzabo}. Also, note that after changing a crossing, the $f$-value at that crossing is multiplied by $-1$. Thus, given a domain $D$ in the original diagram $\mathcal{D}$, the static corners all have the same $f$-value $a_D$, whereas the fixable corners all have the same $f$-value $-a_D$. 

Let $x$ and $y$ be Kauffman states in the diagram $\mathcal{D}$. 
Let $D\in Dom_e(\mathcal{D})$ be a domain which contributes to the $\delta$-gradings $\delta(x)$ and $\delta(y)$ (that is, $D$ is not adjacent to the special marked edge $e$). The contributions to $\delta(x)$ and $\delta(y)$ from the writhe are the same, and we are interested in the difference
 \[ | \delta(x) - \delta(y) | = | \sum_{c \in \text{cr}(\mathcal{D})} f(c_x) - \sum_{c \in \text{cr}(\mathcal{D})} f(c_y) | .
 \]
We can view each summation on the right-hand side as a sum over domains $D \in Dom_e(\mathcal{D})$ via the bijections with $\text{cr}(\mathcal{D})$ given by $x$ and $y$.
Let $c_x$, respectively $c_y$, denote the crossing associated to the domain $D$ by $x$, respectively $y$. The domain $D$  falls into exactly one of four categories:
\begin{enumerate}
\item $c_x \in D$ is a static crossing, $c_y \in D$ is a fixable crossing; 
\item $c_x \in D$ is a fixable crossing, $c_y \in D$ is a static crossing;
\item $c_x \in D$ and $c_y \in D$ are both fixable crossings;
\item $c_x \in D$ and $c_y \in D$ are both static crossings.
\end{enumerate}
In cases (1) and (2), the $f$-values at this domain $D$ for $x$ and $y$ differ by at most $\frac{1}{2}$.  The maximum number of possible domains of the diagram $\mathcal{D}$ that fall into case (1) is bounded above by the number of fixable crossings $n$. Similarly, the number of possible domains in case (2) is bounded above by $n$. In cases (3) and (4), the contributions to the $\delta$-grading of $x$ and $y$ are the same for the domain $D$. Thus, summing over all domains in $Dom_e(\mathcal{D})$, we see
\[ | \delta(x) - \delta(y) | \leq  \frac{1}{2} n + \frac{1}{2} n + 0  + 0  = n.  \]
If we take the maximum of the left-hand side of this inequality over all pairs of homogeneous elements $x, y \in C_{\mathcal{D},e}^\delta$, we see immediately $\text{th}(K) \leq n$.
Taking the minimum $n$, over all choices of a set of fixable crossings in a diagram and over all diagrams for $K$, gives the bound  $\text{th}(K)  \leq \text{dalt}(K)$. 
\end{proof}



\bibliographystyle{amsalpha}
\bibliography{thickness}

\end{document}